\documentclass[10pt, a4paper]{article}

\usepackage{verbatim}
\usepackage{amsfonts}
\usepackage{mathrsfs}
\usepackage{amsmath}
\usepackage{amssymb}
\usepackage{amsthm}
\usepackage{bm}
\usepackage[dvips]{graphicx}
\usepackage[english]{babel}
\usepackage{cite}
\usepackage{caption}
\usepackage{subcaption}
\usepackage[titletoc,toc,title]{appendix}

\usepackage[T1]{fontenc}
\usepackage{txfonts}

\usepackage{xcolor}
\definecolor{Green}{rgb}{0,0.4,0}

\DeclareMathAlphabet{\mathpzc}{OT1}{pzc}{m}{it}

\newcommand{\ud}{\mathrm{d}}

\newcommand{\vol}{\mathrm{vol}}

\newcommand{\Li}{\mathrm{Li}}

\theoremstyle{definition}
\newtheorem{dfnt}{Definition}[section]

\theoremstyle{plain}
\newtheorem{lemma}{Lemma}[section]
\newtheorem{thrm}{Theorem}

\theoremstyle{remark}
\newtheorem{rmk}{Remark}
\newtheorem*{rmk*}{Remark}

\title{The asymptotic volume of diagonal subpolytopes of symmetric stochastic matrices}
\author{
\begin{tabular}[t]{rl} 
J. de Jong\hspace{8mm}  & \hspace{8mm}R. Wulkenhaar \\
\multicolumn{2}{c}{\small WWU M\"unster, Germany}\\
\multicolumn{2}{c}{\small Mathematical Institute}\\
\footnotesize \texttt{j.dejong@uni-muenster.de} & \footnotesize \texttt{raimar@math.uni-muenster.de}
\end{tabular}
}\normalsize
\date{\today}

\begin{document}
\maketitle
\begin{abstract}
The asymptotic volume of the polytope of symmetric stochastic matrices can be determined by asymptotic enumeration techniques as in the case of the Birkhoff polytope. These methods can be extended to polytopes of symmetric stochastic matrices with given diagonal, if this diagonal varies not too wildly. To this end, the asymptotic number of symmetric matrices with natural entries, zero diagonal and varying row sums is determined.
\end{abstract}

\vspace{5mm}
\textrm{\small Keywords: Asymptotic enumeration, Polytope volumes}\\
\textrm{\small MSC 2010: 05A16, 52B11}\normalsize

\section{Introduction}
Convex polytopes arise naturally in various places in mathematics. A fundamental problem is the polytope's volume. Some results are known for low-dimensional setups~\cite{gruber}, polytopes with only a few vertices, or highly symmetric cases~\cite{mckay3, canfield1}. This work belongs to the latter category.\\
\begin{dfnt}
A convex polytope $\mathpzc{P}$ is the convex hull of a finite set $S_{\mathpzc{P}}=\{v_{j}\in\mathbb{R}^{n}\}$ of vertices.
\end{dfnt}
Stochastic matrices are square matrices with nonnegative entries, such that every row of the matrix sums to one. The symmetric stochastic $N\times N$-matrices are an example of a convex polytope. It will be denoted by $\mathcal{P}_{N}$. Its vertices are given by the symmetric permutation matrices. There are $\sum_{j=0}^{N/2}\binom{N}{2j}(2j-1)!!$ such matrices. It follows directly from the Birkhof-Von Neumann theorem that all symmetric stochastic matrices are of this form. A basis for this space is given by
\begin{equation*}
\{I_{N}\}\cup\{B^{(jk)}|1\leq j<k\leq N\}\quad,
\end{equation*}
where $I_{N}$ is the $N\times N$ identity matrix and the matrix elements of $B^{(jk)}$ are given by
\begin{equation*}
B^{(jk)}_{lm}=\left\{\begin{array}{ll}B^{(jk)}_{lm}=1&\text{, if }\{l,m\}=\{j,k\}\quad;\\B^{(jk)}_{lm}=1&\text{, if }j\neq l=m\neq k\quad;\\B^{(jk)}_{lm}=0&\text{, otherwise }\end{array}\right.\quad.
\end{equation*}
All these vertices are linearly independent and it follows that the polytope is\\
\mbox{$\binom{N}{2}$-dimensional}. 
\begin{dfnt}
A convex subpolytope $\mathpzc{P}'$ of a convex polytope $\mathpzc{P}$ is the convex hull of a finite set $\{v'_{j}\in \mathpzc{P}\}$ of elements in $\mathpzc{P}$.
\end{dfnt}
Slicing a polytope yields a surface of section, which is itself a convex space and, hence, a polytope. Determining its vertices is in general very difficult.\\
Spaces of symmetric stochastic matrices with several diagonal entries fixed are examples of such slice subpolytopes of $\mathcal{P}_{N}$, provided that these entries lie between zero and one. The slice subpolytope of $\mathcal{P}_{N}$, obtained by fixing all diagonal entries $h_{j}\in[0,1]$, will be called the diagonal subpolytope $P_{N}(h_{1},\ldots,h_{N})$ here. This is a polytope of dimension $N(N-3)/2$. These polytopes form the main subject of this paper.\\

To keep the notation light, vectors of $N$ elements are usually written by a bold symbol. The diagonal subpolytope with entries $h_{1},\ldots,h_{N}$ will thus be written by $P_{N}(\bm{h})$.\\

The main results are the following two theorems.
\begin{thrm}\label{thrm:p1}
Let $V_{N}(\bm{t};\lambda)$ be the number of symmetric $N\times N$-matrices with an empty diagonal and entries in the natural numbers such that $t_{j}$ is the $j$-th row sum. Denote the total entry sum by $x=\sum_{j=1}^{N}t_{j}$ and let $\lambda$ be the average matrix entry 
\begin{equation*}
\lambda=\frac{x}{N(N-1)}>\frac{C}{\log N}\quad.
\end{equation*}
If for some $\omega\in(0,\frac{1}{4})$ the limit
\begin{equation*}
\lim_{N\rightarrow\infty}\frac{t_{j}-\lambda(N-1)}{\lambda N^{\frac{1}{2}+\omega}}=0\qquad\text{for all }j=1,\ldots,N\quad,
\end{equation*}
then the number of such matrices is asymptotically ($N\rightarrow\infty$) given by
\begin{align*}
&\hspace{-4mm}V_{N}(\bm{t};\lambda)=\frac{\sqrt{2}(1+\lambda)^{\binom{N}{2}}}{(2\pi\lambda(\lambda+1)N)^{\frac{N}{2}}}\big(1+\frac{1}{\lambda}\big)^{\frac{x}{2}}\exp[\frac{14\lambda^{2}+14\lambda-1}{12\lambda(\lambda+1)}]\\
&\times\exp[\frac{-1}{2\lambda(\lambda+1)N}\sum_{m}(t_{m}-\lambda(N-1))^{2}]\exp[\frac{-1}{\lambda(\lambda+1)N^{2}}\sum_{m}(t_{m}-\lambda(N-1))^{2}]\\
&\times\exp[\frac{2\lambda+1}{6\lambda^{2}(\lambda+1)^{2}N^{2}}\sum_{m}(t_{m}-\lambda(N-1))^{3}]\exp[-\frac{3\lambda^{2}+3\lambda+1}{12\lambda^{3}(\lambda+1)^{3}N^{3}}\sum_{m}(t_{m}-\lambda(N-1))^{4}]\\
&\times\exp[\frac{1}{4\lambda^{2}(\lambda+1)^{2}N^{4}}\big(\sum_{m}(t_{m}-\lambda(N-1))^{2}\big)^{2}]\times\Big(1+\mathcal{O}(N^{-\frac{1}{2}+6\omega})\Big)\quad.
\end{align*}
\end{thrm}

\begin{thrm}\label{thrm:p2}
Let $\bm{h}=h_{1},\ldots,h_{N}$ with $h_{j}\in[0,1]$ and $\chi=\sum_{j=1}^{N}h_{j}$. If
\begin{equation*}
\lim_{N\rightarrow\infty}N^{\frac{1}{2}-\omega}\frac{N-1}{N-\chi}\cdot \big|h_{j}-\frac{\chi}{N}\big|=0\quad\text{for all }j=1,\ldots,N\qquad,
\end{equation*}
and for some $\omega\in(\frac{\log\log N}{2\log N},\frac{1}{4})$, then the asymptotic volume ($N\rightarrow\infty$) of the polytope of symmetric stochastic $N\times N$-matrices with diagonal $(h_{1},\ldots,h_{N})$ is given by
\begin{align*}
&\hspace{-4mm}\vol(P_{N}(\bm{h}))=\sqrt{2}e^{\frac{7}{6}}\Big(\frac{e(N-\chi)}{N(N-1)}\Big)^{\binom{N}{2}}\Big(\frac{N(N-1)^{2}}{2\pi(N-\chi)^{2}}\Big)^{\frac{N}{2}}\exp[-\frac{N(N-1)^{2}}{2(N-\chi)^{2}}\sum_{j}(h_{j}-\frac{\chi}{N})^{2}]\nonumber\\
&\times\exp[-\frac{(N-1)^{2}}{(N-\chi)^{2}}\sum_{j}(h_{j}-\frac{\chi}{N})^{2}]\exp[-\frac{N(N-1)^{3}}{3(N-\chi)^{3}}\sum_{j}(h_{j}-\frac{\chi}{N})^{3}]\nonumber\\
&\times\exp[-\frac{N(N-1)^{4}}{4(N-\chi)^{4}}\sum_{j}(h_{j}-\frac{\chi}{N})^{4}]\exp[\frac{(N-1)^{4}}{4(N-\chi)^{4}}\big(\sum_{j}(h_{j}-\frac{\chi}{N})^{2}\big)^{2}]\\
&\times\big(1+\mathcal{O}(N^{-\frac{1}{2}+6\omega})\big)\quad.
\end{align*}
\end{thrm}

The outline of this paper is as follows. In Paragraph \ref{sec:cp} the volume problem is formulated as a counting problem and subsequently as a contour integral. Under the assumption of a restricted region this is subsequently integrated in Paragraph \ref{sec:icp}. Paragraph \ref{sec:rir} is dedicated to a fundamental lemma to actually restrict the integration region. The volume of the diagonal subpolytopes is extracted from the counting result in Paragraph \ref{sec:pv}.

\section{Counting problem\label{sec:cp}}
The volume of a polytope $\mathpzc{P}$ in $\mathbb{R}^{n}$ with basis $\{\mathcal{B}_{j}\in\mathbb{R}^{n}|1\leq j\leq d\}$ is obtained by
\begin{equation*}
\int_{[0,1]^{d}}\ud\bm{u}\,\bm{1}_{\mathpzc{P}}(\sum_{j=1}^{d}u_{j}\mathcal{B}_{j})\quad,
\end{equation*}
where $\bm{1}_{\mathpzc{P}}$ is the indicator function for the polytope $\mathpzc{P}$. If the polytope is put on a lattice $(a\mathbb{Z})^{n}$ with lattice parameter $a\in(0,1)$, an approximation of this volume is obtained by counting the lattice sites inside the polytope and multiplying this by the volume $a^{n}$ of a single cell. This approximation becomes better as the lattice parameter shrinks. In the limit this yields
\begin{equation}
\vol(\mathpzc{P})=\lim_{a\rightarrow 0} a^{n}\,\,|\{\mathpzc{P}\cap (a\mathbb{Z})^{n}\}|\quad.\label{e:ehrhart}
\end{equation}
This approach is formalized by the Ehrhart polynomial~\cite{stanley1}, which counts the number of lattice sites of $\mathbb{Z}^{n}$ in a dilated polytope. A dilation of a polytope $\mathpzc{P}$ by a factor $a^{-1}>1$ yields the polytope $a^{-1}\mathpzc{P}$, which is the convex hull of the dilated vertices $S_{a^{-1}\mathpzc{P}}=\{a^{-1}v|v\in S_{\mathpzc{P}}\}$. That the obtained volume is the same, follows from the observation
\begin{equation*}
|\{a^{-1}\mathpzc{P}\cap \mathbb{Z}^{n}\}|=|\{\mathpzc{P}\cap (a\mathbb{Z})^{n}\}|\quad.
\end{equation*}

The volume integral of the diagonal subpolytope $P_{N}(\bm{h})$ is
\begin{equation*}
\vol(P_{N}(\bm{h}))=\Big\{\!\!\prod_{1\leq k<l\leq N}\int_{0}^{1}\ud u_{kl}\Big\}\quad\bm{1}_{P_{N}(\bm{h})}\big(I_{N}+\!\!\!\!\sum_{1\leq k<l\leq N}\!\!u_{kl}(B^{(kl)}-I_{N})\big)\quad.
\end{equation*}
To see that this integral covers the polytope, it suffices to see that the any symmetric stochastic matrix $A=(a_{kl})$ is decomposed in basis vectors as
\begin{equation*}
A=\big(a_{kl}\big)=I_{N}+\sum_{1\leq k<l\leq N}a_{kl}(B^{(kl)}-I_{N})\quad.
\end{equation*}

The next step is to introduce a lattice $(a\mathbb{Z})^{\binom{N}{2}}$ and count the sites inside the polytope. Each such site is a symmetric stochastic matrix with $h_{1},\ldots,h_{N}$ on the diagonal.\\
Since the volume depends continuously on the extremal points, it can be assumed without loss of generality that all $h_{j}$ are rational. This implies that a dilation factor $a^{-1}$ exists, such that all $a^{-1}(1-h_{j})=t_{j}\in\mathbb{N}$ and that the matrices that solve
\begin{equation}
\left(\begin{array}{cccc}0&b_{12}&\cdots&b_{1N}\\b_{12}&0&\cdots&b_{2N}\\\vdots&\vdots&\ddots&\vdots\\b_{1N}&b_{2N}&\cdots&0\end{array}\right)\left(\begin{array}{c}1\\1\\\vdots\\1\end{array}\right)=\left(\begin{array}{c}t_{1}\\t_{2}\\\vdots\\t_{N}\end{array}\right)\label{e:natmateq}
\end{equation}
with $t_{j},b_{jk}\in\mathbb{N}$ are to be counted. This yields a number $V_{N}(\bm{t})$. The polytope volume is then given by
\begin{equation*}
\vol(P_{N}(\bm{h}))=\lim_{a\rightarrow0}\,a^{\frac{N(N-3)}{2}}V_{N}(\frac{1-h_{1}}{a},\ldots,\frac{1-h_{N}}{a})\quad,
\end{equation*}
where 
\begin{align}
&\hspace{-4mm}V_{N}(\bm{t})=\oint_{\mathcal{C}}\frac{\ud w_{1}}{2\pi i w_{1}^{1+t_{1}}}\ldots \oint_{\mathcal{C}}\frac{\ud w_{N}}{2\pi i w_{N}^{1+t_{N}}}\,\prod_{1\leq k<l\leq N}\frac{1}{1-w_{k}w_{l}}\label{e:V2}\quad.
\end{align}
To see this, let the possible values $m$ for the matrix element $b_{jk}$ be given by the generating function
\begin{equation*}
\frac{1}{1-w_{j}w_{k}}=\sum_{m=0}^{\infty}(w_{j}w_{k})^{m}\quad.
\end{equation*}
Applying this to all matrix entries shows that $V_{N}(\bm{t})$ is given by the coefficient of the term $w_{1}^{t_{1}}w_{2}^{t_{2}}\ldots w_{N}^{t_{N}}$ in $\prod_{1\leq j<k\leq N}\frac{1}{1-w_{j}w_{k}}$. Formulating this in derivatives yields
\begin{align*}
&\hspace{-4mm}V_{N}(\bm{t})=\frac{1}{t_{1}!}\frac{\ud}{\ud w_{1}}\Big|_{w_{1}=0}^{t_{1}}\ldots \frac{1}{t_{N}!}\frac{\ud}{\ud w_{N}}\Big|_{w_{N}=0}^{t_{N}}\;\prod_{1\leq k<l\leq N}\frac{1}{1-w_{k}w_{l}}\quad.
\end{align*}
By Cauchy's integral formula the number of matrices (\ref{e:V2}) follows from this. The contour $\mathcal{C}$ encircles the origin once in the positive direction, but not the pole at $w_{k}w_{l}=1$.\\

The next step is to parametrize this contour explicitly and find a way to compute the integral for $N\rightarrow\infty$. This must be done in such a way that a combinatorial treatments is avoided. A convenient choice is
\begin{equation}
w_{j}=\sqrt{\frac{\lambda_{j}}{\lambda_{j}+1}}e^{i\varphi_{j}}\qquad\text{, with }\lambda_{j}\in\mathbb{R}_{+}\text{ and }\varphi_{j}\in[-\pi,\pi)\quad.\label{e:conpar}
\end{equation}
Later a specific value for $\lambda_{j}$ will be chosen.\\
The counting problem has now been turned into an integral over the $N$-dimensional torus
\begin{align}
&\hspace{-4mm}V_{N}(\bm{t})=\Big(\prod_{j=1}^{N}(1+\frac{1}{\lambda_{j}})^{\frac{t_{j}}{2}}\Big)(2\pi)^{-N}\int_{\mathbb{T}^{N}}\!\!\!\ud\bm{\varphi}\,e^{-i\sum_{j=1}^{N}\varphi_{j}t_{j}}\nonumber\\
&\times\prod_{1\leq k<l\leq N}\!\!\frac{\sqrt{(1+\lambda_{k})(1+\lambda_{l})}}{\sqrt{(1+\lambda_{k})(1+\lambda_{l})}-\sqrt{\lambda_{k}\lambda_{l}}}\,\frac{1}{1-\frac{\sqrt{\lambda_{k}\lambda_{l}}}{\sqrt{(1+\lambda_{k})(1+\lambda_{l})}-\sqrt{\lambda_{k}\lambda_{l}}} (e^{i(\varphi_{k}+\varphi_{l})}-1)}\quad,\label{e:V3}
\end{align}
where we have written $\ud\bm{\varphi}$ for $\ud\varphi_{1}\ldots \ud\varphi_{N}$.\\

The notations
\begin{equation*}
x=\sum_{j}t_{j}=\sum_{j=1}^{N}t_{j}\qquad\text{and}\qquad \sum_{k<l}(\varphi_{k}+\varphi_{l})=\sum_{1\leq k<l\leq N}(\varphi_{k}+\varphi_{l})
\end{equation*}
are used, when no doubt about $N$ can exist. When no summation bounds are mentioned, these will always be $1$ and $N$. The notation $a\ll b$ indicates that $a<b$ and $a/b\rightarrow0$.\\

The main tool for these integrals will be the stationary phase method, also called the saddle-point method. In the form used in this paper, the exponential of a function $f$ is integrated around its maximum $\tilde{x}$, so that
\begin{align}
&\hspace{-4mm}\lim_{\Lambda\rightarrow\infty}\int\!\ud x\,e^{\Lambda f(x)}=\lim_{\Lambda\rightarrow\infty}\exp[\Lambda f(\tilde{x})]\int\!\ud x\,\exp[\frac{\Lambda f^{(2)}(\tilde{x})}{2}(x\!-\!\tilde{x})^{2}+\frac{\Lambda f^{(3)}(\tilde{x})}{6}(x\!-\!\tilde{x})^{3}]\nonumber\\
&\times\exp[\frac{\Lambda f^{(4)}(\tilde{x})}{24}(x\!-\!\tilde{x})^{4}]\nonumber\\
&=\exp[\Lambda f(\tilde{x})]\sqrt{\frac{-2\pi}{\Lambda f^{(2)}(\tilde{x})}}\Big(1+\frac{15}{16}\frac{2(f^{(3)}(\tilde{x}))^{2}}{9\Lambda(-f^{(2)}(\tilde{x}))^{3}}+\frac{3}{4}\frac{f^{(4)}(\tilde{x})}{6\Lambda (f^{(2)}(\tilde{x}))^{2}}+\mathcal{O}(\Lambda^{-2})\Big)\quad.\label{e:spm}
\end{align}

Many counting problems can be computed asymptotically by the saddle-point method~\cite{mckay1,canfield2}. Often it is assumed that all $t_{j}$ are equal, but we show that it suffices to demand that they do not deviate too much from this symmetric case.

\section{Integrating the central part\label{sec:icp}}

The integrals in (\ref{e:V3}) are too difficult to compute in full generality. A useful approximation can be obtained from the observation that the integrand
\begin{equation}
\big|\frac{1}{1-\mu(e^{iy}-1)}\big|^{2}=\frac{1}{1-2\mu(\mu+1)(\cos(y)-1)}\qquad\text{for }y\in(-2\pi,2\pi)\label{e:intfac}
\end{equation}
is concentrated in a neighbourhood of the origin and the antipode $y=\pm2\pi$, where it takes the value $1$. This is plotted in Figure \ref{f:estplot}. For small $y$ and $\mu y$ the absolute value of the integrand factor can be written as
\begin{equation}
\big|\frac{1}{1-\mu(e^{iy}-1)}\big|=\sqrt{\frac{1}{1+\mu(\mu+1)y^{2}}}\big(1+\mathcal{O}(y^{4})\big)\,\quad.\label{e:fracest}
\end{equation}

\begin{figure}[!hb]\centering
\includegraphics[width=0.7\textwidth]{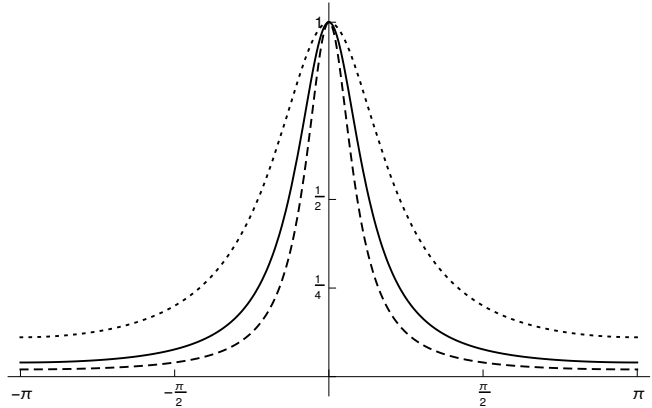}
\caption{The absolute value squared of the integrand factor (\ref{e:intfac}) for $\mu=1,2$ and $3$ in dotted, continuous and dashed lines respectively.\label{f:estplot}}
\end{figure}

It is concentrated in a small region around the origin and the antipode. The form of the region is assumed to be $[\delta_{N},\delta_{N}]^{N}$ with 
\begin{equation*}
\delta_{N}=\frac{N^{-\alpha}\zeta_{N}}{\min_{j}\{\lambda_{j}\}}\quad,
\end{equation*}
where $\alpha\in(0,1/2)$ and $\zeta_{N}$ tends slowly to infinity. In the remainder of this paragraph the integral inside this box will be computed.\\
To this end, a lower bound is introduced. Below this threshold we do not strive for accuracy. The aim is thus to find the asymptotic number $V_{N}(\bm{t})$ for configurations $\bm{t}$, such that this number is larger than the Lower bound.
\begin{dfnt}\emph{Lower bound}\label{dfnt:lb}\\
For $N$, $\alpha\in(0,1/2)$, $t_{j}\in\mathbb{N}$ and $\lambda_{j}\in\mathbb{R}_{+}^{N}$ for $j=1,\ldots,N$ we define the Lower bound by
\begin{align*}
&\hspace{-4mm}\mathcal{E}_{\alpha} = (2\pi\lambda(\lambda+1)N)^{-\frac{N}{2}}\Big(\prod_{j=1}^{N}(1+\frac{1}{\lambda_{j}})^{\frac{t_{j}}{2}}\Big)\Big(\prod_{k<l}\frac{\sqrt{(1+\lambda_{k})(1+\lambda_{l})}}{\sqrt{(1+\lambda_{k})(1+\lambda_{l})}-\sqrt{\lambda_{k}\lambda_{l}}}\Big)\\
&\times\exp[\frac{14\lambda^{2}+14\lambda-1}{12\lambda(\lambda+1)}]\exp[-N^{1-2\alpha}]\quad,
\end{align*}
where $\lambda=N^{-1}\sum_{j}\lambda_{j}$.
\end{dfnt}

The integral in $[-\delta_{N},\delta_{N}]^{N}$ can now be cast into a simpler form, where the size $\delta_{N}$ of this box can be used as an expansion parameter. The expansion used is
\begin{equation}
\frac{1}{1-\mu(\exp[iy]-1)}=\exp[\sum_{j=1}^{k}A_{j}(iy)^{j}]+\mathcal{O}(y^{k+1}(1+\mu)^{k+1})\label{e:fracexpa}\quad.
\end{equation}
The coefficients $A_{j}(\mu)$ (or $A_{j}$ if the argument is clear) are polynomials in $\mu$ of degree $j$. They are obtained as the polylogarithms
\begin{equation*}
A_{n}(\mu)=\frac{(-1)^{n}}{n!}\; \Li_{1-n}(1+\frac{1}{\mu})\quad.
\end{equation*}
The first four coefficients are
\begin{align}
&A_{1}=\mu\quad;\quad A_{2}=\frac{\mu}{2}(\mu+1)\quad;\quad A_{3}=\frac{\mu}{6}(\mu+1)(2\mu+1)\quad\nonumber\\
&\text{and}\quad A_{4}=\frac{\mu}{24}(\mu+1)(6\mu^{2}+6\mu+1)\quad.\label{e:coef}
\end{align}
The value of the parameter $\mu$ in the above formules can be approximated. Assuming that $\varepsilon_{k}$ is small compared to $\lambda$ and writing $\varepsilon=\max_{k}\varepsilon_{k}$, this is
\begin{align*}
&\hspace{-4mm}\frac{\sqrt{(\lambda+\varepsilon_{k})(\lambda+\varepsilon_{l})}}{\sqrt{(\lambda+\varepsilon_{k}+1)(\lambda+\varepsilon_{l}+1)}-\sqrt{(\lambda+\varepsilon_{k})(\lambda+\varepsilon_{l})}}\approx\lambda+\frac{\varepsilon_{k}+\varepsilon_{l}}{2}\\
&-\frac{2\lambda+1}{8\lambda(\lambda+1)}(\varepsilon_{k}-\varepsilon_{l})^{2}+\frac{2\lambda^{2}+2\lambda+1}{16\lambda^{2}(\lambda+1)^{2}}(\varepsilon_{k}^{3}-\varepsilon_{k}^{2}\varepsilon_{l}-\varepsilon_{k}\varepsilon_{l}^{2}+\varepsilon_{l}^{3})+\mathcal{O}(\frac{\varepsilon^{4}}{\lambda^{3}})\quad.
\end{align*}

Applying this in combination with (\ref{e:fracexpa}) produces the combinations
\begin{align}
&\hspace{-4mm}\sum_{k<l}(\varphi_{k}+\varphi_{l})\cdot(\frac{\sqrt{\lambda_{k}\lambda_{l}}}{\sqrt{(1+\lambda_{k})(1+\lambda_{l})}-\sqrt{\lambda_{k}\lambda_{l}}})\nonumber\\
&=\sum_{j=1}^{N}\varphi_{j}\big[\frac{N-2}{2}\lambda_{j}+\frac{N}{2}\lambda-B_{1}\big(N\varepsilon_{j}^{2}+\sum_{m}\varepsilon_{m}^{2}\big)+C_{1}\big(N\varepsilon_{j}^{3}-\varepsilon_{j}\sum_{m}\varepsilon_{m}^{2}+\sum_{m}\varepsilon_{m}^{3}\big)\big]\nonumber\\
&\times(1+\mathcal{O}(N\frac{\varepsilon^{4}}{\lambda^{4}}))\qquad;\nonumber\\
&\hspace{-4mm}\sum_{k<l}(\varphi_{k}+\varphi_{l})^{2}\cdot A_{2}(\frac{\sqrt{\lambda_{k}\lambda_{l}}}{\sqrt{(1+\lambda_{k})(1+\lambda_{l})}-\sqrt{\lambda_{k}\lambda_{l}}})\nonumber\\
&=\big[\sum_{j=1}^{N}\varphi_{j}^{2}\Big((N-2)A_{2}+\varepsilon_{j}B_{2}(N-4)-(N-4)C_{2}\varepsilon_{j}^{2}-C_{2}\sum_{m}\varepsilon_{m}^{2}\Big)\nonumber\\
&+\sum_{j=1}^{N}\varphi_{j}\Big(A_{2}\sum_{m}\varphi_{m}+2B_{2}\varepsilon_{j}\sum_{m}\varphi_{m}-2C_{2}\varepsilon_{j}^{2}\sum_{m}\varphi_{m}+D_{2}\varepsilon_{j}\sum_{m}\varepsilon_{m}\varphi_{m}\Big)\big]\nonumber\\
&\times(1+\mathcal{O}(N\frac{\varepsilon^{3}}{\lambda^{3}}))\qquad;\nonumber\\
&\hspace{-4mm}\sum_{k<l}(\varphi_{k}+\varphi_{l})^{3}\cdot A_{3}(\frac{\sqrt{\lambda_{k}\lambda_{l}}}{\sqrt{(1+\lambda_{k})(1+\lambda_{l})}-\sqrt{\lambda_{k}\lambda_{l}}})\nonumber\\
&=\big[\sum_{j}\varphi_{j}^{3}A_{3}(N-4)+3A_{3}\sum_{j}\varphi_{j}^{2}\sum_{m}\varphi_{m}\big]\times(1+\mathcal{O}(\frac{\varepsilon}{\lambda}))\qquad\text{and}\nonumber\\
&\hspace{-4mm}\sum_{k<l}(\varphi_{k}+\varphi_{l})^{4}\cdot A_{4}(\frac{\sqrt{\lambda_{k}\lambda_{l}}}{\sqrt{(1+\lambda_{k})(1+\lambda_{l})}-\sqrt{\lambda_{k}\lambda_{l}}})\nonumber\\
&=\big[\sum_{j}\varphi_{j}^{4}A_{4}(N-8)+4A_{4}\sum_{j}\varphi_{j}^{3}\sum_{m}\varphi_{m}+3A_{4}(\sum_{j}\varphi_{j}^{2})^{2}\big]\times(1+\mathcal{O}(\frac{\varepsilon}{\lambda}))\quad.\label{e:comb}
\end{align}

Here we used the additional combinations
\begin{align}
&B_{1} = \frac{2\lambda+1}{8\lambda(\lambda+1)}\quad;\quad C_{1}=\frac{2\lambda^{2}+2\lambda+1}{16\lambda^{2}(\lambda+1)^{2}}\quad;\quad B_{2}=\frac{2\lambda+1}{4}\nonumber\\
&C_{2}=\frac{2\lambda^{2}+2\lambda+1}{16\lambda(\lambda+1)}\quad;\quad D_{2}=\frac{6\lambda^{2}+6\lambda+1}{8\lambda(\lambda+1)}\label{e:coef2}
\end{align}
to simplify the notation.\\

The simplest way to compute this integral is to ensure that the linear part of the exponent is small.
Splitting $\lambda_{j}=\lambda+\varepsilon_{j}$ and choosing the value 
\begin{align*}
&\hspace{-4mm}\varepsilon_{j}=\frac{2}{N-2}\big(t_{j}-\lambda(N-1)\big)
\end{align*}
is done therefore. Combined with the assumption that $x=\sum_{j}t_{j}=\lambda N(N-1)$, this implies that $\sum_{m}\varepsilon_{m}=0$. Assuming furthermore that $|t_{j}-\lambda(N-1)|\ll \lambda N^{\frac{1}{2}+\omega}$, the error terms $|\varepsilon/\lambda|\ll N^{-\frac{1}{2}+\omega}$ follow.\\

The first step now is to focus on the integral inside the box $[-\delta_{N},\delta_{N}]^{N}$, simplify and calculate this.
\begin{rmk}
The estimates in Lemma \ref{l:funris} cause the integral (\ref{e:V3}) to depend non-trivially on $\lambda$. For that reason $\lambda$ is explicitly mentioned as an argument.
\end{rmk}
\begin{lemma}\label{l:diff}
Assume that $K,N\in\mathbb{N}$, $\omega,\alpha\in\mathbb{R}_{+}$ are chosen such that $\omega\in(0,\frac{\log(K\alpha-2) + \log\log N}{4\log N})$, $\alpha\in(0,\frac{1}{4}-\omega)$ and $K>2/\alpha+1$. Define
\begin{equation*}
\delta_{N}=\frac{N^{-\alpha}\zeta_{N}}{\min \{\lambda_{j}\}}\quad,
\end{equation*}
so that $\zeta_{N}\rightarrow\infty$ and $N^{-\delta}\zeta_{N}\rightarrow0$ for any $\delta>0$, when $N\rightarrow\infty$. If $x=\sum_{j}t_{j}$, the average matrix entry $\lambda=\frac{x}{N(N-1)}$ and
\begin{equation*}
\lim_{N\rightarrow\infty}\frac{t_{j}-\lambda(N-1)}{\lambda N^{\frac{1}{2}+\omega}}=0\quad\text{for }j=1,\ldots,N\qquad,
\end{equation*}
then the integral
\begin{align*}
&\hspace{-4mm}V_{N}(\bm{t})=\Big(\prod_{j=1}^{N}(1+\frac{1}{\lambda_{j}})^{\frac{t_{j}}{2}}\Big)(2\pi)^{-N}\int_{[-\delta_{N},\delta_{N}]^{N}}\!\!\!\ud\bm{\varphi}\,e^{-i\sum_{j=1}^{N}\varphi_{j}t_{j}}\\
&\times\prod_{1\leq k<l\leq N}\!\!\frac{\sqrt{(1+\lambda_{k})(1+\lambda_{l})}}{\sqrt{(1+\lambda_{k})(1+\lambda_{l})}-\sqrt{\lambda_{k}\lambda_{l}}}\,\frac{1}{1-\frac{\sqrt{\lambda_{k}\lambda_{l}}}{\sqrt{(1+\lambda_{k})(1+\lambda_{l})}-\sqrt{\lambda_{k}\lambda_{l}}} (e^{i(\varphi_{k}+\varphi_{l})}-1)}
\end{align*}
is given by
\begin{align*}
&\hspace{-4mm}V_{N}(\bm{t};\lambda)=\frac{2}{(2\pi)^{N}}\Big(\prod_{j=1}^{N}(1+\frac{1}{\lambda_{j}})^{\frac{t_{j}}{2}}\Big)\cdot\Big(\prod_{1\leq k<l\leq N}\!\!\frac{\sqrt{(1+\lambda_{k})(1+\lambda_{l})}}{\sqrt{(1+\lambda_{k})(1+\lambda_{l})}-\sqrt{\lambda_{k}\lambda_{l}}}\Big)\\
&\times\int_{[-\delta_{N},\delta_{N}]^{N}}\!\!\!\!\!\!\!\!\!\!\!\!\!\ud\bm{\varphi}\,\exp[-i\!\sum_{j}\!\varphi_{j}t_{j}]\exp[\sum_{n=1}^{K-1}\!i^{n}\!\sum_{k<l}\!A_{n}\big(\frac{\sqrt{\lambda{k}\lambda_{l}}}{\sqrt{(1\!+\!\lambda{k})(1\!+\!\lambda_{l})}-\sqrt{\lambda{k}\lambda_{l}}}\big)\cdot\big(\varphi_{k}\!+\!\varphi_{l}\big)^{n}]+\mathcal{D}\,,
\end{align*}
up to a difference $\mathcal{D}$ that satisfies
\begin{align*}
&\hspace{-4mm}|\mathcal{D}|\leq \mathcal{O}(N^{2-K\alpha})\frac{\sqrt{2}(1+\lambda)^{\binom{N}{2}}}{(2\pi\lambda(\lambda+1)N)^{\frac{N}{2}}}\big(1+\frac{1}{\lambda}\big)^{\frac{x}{2}}\exp[\frac{10\lambda^{2}+10\lambda+1}{4\lambda(\lambda+1)}]\\
&\exp[\frac{-1}{2\lambda(\lambda+1)N}\sum_{m}(t_{m}-\lambda(N-1))^{2}]\exp[\frac{3}{4\lambda^{2}(\lambda+1)^{2}N^{2}}\sum_{m}(t_{m}-\lambda(N-1))^{2}]\\
&\times\exp[\frac{2\lambda+1}{6\lambda^{2}(\lambda+1)^{2}N^{2}}\sum_{m}(t_{m}-\lambda(N-1))^{3}]\exp[\frac{6\lambda^{2}+6\lambda+1}{24\lambda^{3}(\lambda+1)^{3}N^{3}}\sum_{m}(t_{m}-\lambda(N-1))^{4}]\\
&\times\exp[\frac{6\lambda^{2}+6\lambda+1}{8\lambda^{3}(\lambda+1)^{3}N^{4}}\big(\sum_{m}(t_{m}-\lambda(N-1))^{2}\big)^{2}]\quad.
\end{align*}
\end{lemma}
\begin{proof}
To the fraction 
\begin{equation*}
\Big(1-\frac{\sqrt{\lambda_{k}\lambda_{l}}}{\sqrt{(1+\lambda_{k})(1+\lambda_{l})}-\sqrt{\lambda_{k}\lambda_{l}}} (e^{i(\varphi_{k}+\varphi_{l})}-1)\Big)^{-1}
\end{equation*}
in the integral (\ref{e:V3}) the expansion (\ref{e:fracexpa}) in combination with (\ref{e:coef}) and (\ref{e:coef2}) is applied. To prove that contributions in (\ref{e:fracexpa}) of $K$-th order or higher are irrelevant, we put these in the exponential $\exp[h(x)]$. To estimate their contribution, the estimate
\begin{equation*}
|\int\ud x\,e^{f(x)}(e^{h(x)}-1)|\leq \mathcal{O}(\sup_{x} |e^{h(x)}-1|)\cdot \int \ud x\, |e^{f(x)}|
\end{equation*}
is applied to the integral. Taking the absolute value of the integrand sets the imaginary parts of the exponential to zero. In terms of (\ref{e:coef}) and (\ref{e:coef2}) this means that $A_{3}$, $B_{1}$ and $C_{1}$ are set to zero. This integral is calculated in Lemma \ref{l:labval}. Taking this result and setting these coefficients to zero completes the proof.
\end{proof}

\begin{lemma}\label{l:labval}
Assume that $K,N\in\mathbb{N}$, $\omega,\alpha\in\mathbb{R}_{+}$ are chosen such that $\omega\in(0,\frac{\log(K\alpha-2) + \log\log N}{4\log N})$, $\alpha\in(0,\frac{1}{4}-\omega)$ and $K>2/\alpha+1$. Define
\begin{equation*}
\delta_{N}=\frac{N^{-\alpha}\zeta_{N}}{\min \{\lambda_{j}\}}\quad,
\end{equation*}
so that $\zeta_{N}\rightarrow\infty$ and $N^{-\delta}\zeta_{N}\rightarrow0$ for any $\delta>0$, when $N\rightarrow\infty$. If $x=\sum_{j}t_{j}$, the average matrix entry $\lambda=\frac{x}{N(N-1)}>\frac{C}{\log N}$ and
\begin{equation}
\lim_{N\rightarrow\infty}\frac{t_{j}-\lambda(N-1)}{\lambda N^{\frac{1}{2}+\omega}}=0\quad\text{for }j=1,\ldots,N\qquad,
\end{equation}
then the integral
\begin{align*}
&\hspace{-4mm}V_{N}(\bm{t};\lambda)=\frac{2}{(2\pi)^{N}}\Big(\prod_{j=1}^{N}(1+\frac{1}{\lambda_{j}})^{\frac{t_{j}}{2}}\Big)\cdot\Big(\prod_{1\leq k<l\leq N}\!\!\frac{\sqrt{(1+\lambda_{k})(1+\lambda_{l})}}{\sqrt{(1+\lambda_{k})(1+\lambda_{l})}-\sqrt{\lambda_{k}\lambda_{l}}}\Big)\\
&\times\int_{[-\delta_{N},\delta_{N}]^{N}}\!\!\!\!\!\!\!\!\!\!\!\!\!\ud\bm{\varphi}\,\exp[-i\!\sum_{j}\!\varphi_{j}t_{j}]\exp[\sum_{n=1}^{K-1}\!i^{n}\!\sum_{k<l}\!A_{n}\big(\frac{\sqrt{\lambda{k}\lambda_{l}}}{\sqrt{(1\!+\!\lambda{k})(1\!+\!\lambda_{l})}-\sqrt{\lambda{k}\lambda_{l}}}\big)\cdot\big(\varphi_{k}\!+\!\varphi_{l}\big)^{n}]
\end{align*}
is asymptotically ($N\rightarrow\infty$) given by
\begin{align*}
&\hspace{-4mm}V_{N}(\bm{t};\lambda)=\frac{\sqrt{2}}{(2\pi\lambda(\lambda+1)N)^{\frac{N}{2}}}\big[\prod_{n}(1+\frac{1}{\lambda_{n}})^{\frac{t_{n}}{2}}\big]\big[\prod_{k<l}\frac{\sqrt{(1+\lambda_{k})(1+\lambda_{l})}}{\sqrt{(1+\lambda_{k})(1+\lambda_{l})}-\sqrt{\lambda_{k}\lambda_{l}}}\big]\\
&\times\exp[\frac{14\lambda^{2}+14\lambda-1}{12\lambda(\lambda+1)}]\exp[\frac{\sum_{m}\varepsilon_{m}^{2}}{16\lambda^{2}(\lambda+1)^{2}}]\exp[-\frac{(2\lambda+1)^{2}}{128\lambda^{3}(\lambda+1)^{3}}(\sum_{m}\varepsilon_{m}^{2})^{2}]\\
&\times\exp[-\frac{(2\lambda+1)^{2}N}{128\lambda^{3}(\lambda+1)^{3}}\sum_{m}\varepsilon_{m}^{4}]\times\Big(1+\mathcal{O}(N^{-\frac{1}{2}+6\omega}+N^{2+\frac{1}{3C}-K\alpha}\exp[N^{4\omega}])\Big)\\
&=\frac{\sqrt{2}(1+\lambda)^{\binom{N}{2}}}{(2\pi\lambda(\lambda+1)N)^{\frac{N}{2}}}\big(1+\frac{1}{\lambda}\big)^{\frac{x}{2}}\exp[\frac{14\lambda^{2}+14\lambda-1}{12\lambda(\lambda+1)}]\\
&\exp[\frac{-1}{2\lambda(\lambda+1)N}\sum_{m}(t_{m}-\lambda(N-1))^{2}]\exp[\frac{-1}{\lambda(\lambda+1)N^{2}}\sum_{m}(t_{m}-\lambda(N-1))^{2}]\\
&\times\exp[\frac{2\lambda+1}{6\lambda^{2}(\lambda+1)^{2}N^{2}}\sum_{m}(t_{m}-\lambda(N-1))^{3}]\exp[-\frac{3\lambda^{2}+3\lambda+1}{12\lambda^{3}(\lambda+1)^{3}N^{3}}\sum_{m}(t_{m}-\lambda(N-1))^{4}]\\
&\times\exp[\frac{1}{4\lambda^{2}(\lambda+1)^{2}N^{4}}\big(\sum_{m}(t_{m}-\lambda(N-1))^{2}\big)^{2}]\times\Big(1+\mathcal{O}(N^{-\frac{1}{2}+6\omega}+N^{2+\frac{1}{3C}-K\alpha}\exp[N^{4\omega}])\Big)\quad.
\end{align*}
This is much larger than the Lower bound from Definition \ref{dfnt:lb}
\begin{equation*}
\frac{V_{N}(\bm{t};\lambda)}{\mathcal{E}_{\alpha}}\rightarrow\infty\quad.
\end{equation*}
\end{lemma}
\begin{proof}
Define $\varepsilon_{j}=\frac{2}{N-2}\big(t_{j}-\lambda(N-1)\big)$ and assume that $|\varepsilon_{j}|\leq \lambda N^{-\frac{1}{2}+\omega}$ with $0<\omega<1/14$. It follows that $\sum_{j}\varepsilon_{j}=0$.\\

To the integral $V_{N}(\bm{t};\lambda)$ the expansion (\ref{e:fracexpa}) for $k=4$ in combination with (\ref{e:coef}) and (\ref{e:coef2}) is applied. It will follow automatically that the higher orders ($K>5$) in this expansion will yield asymptotically irrelevant factors. This expansion produces the combinations (\ref{e:comb}). 
Introducing $\delta$-functions for $S_{1}=\sum_{m}\varphi_{m}$, $S_{2}=\sum_{m}\varphi_{m}^{2}$, $T_{3}=\sum_{m}\varepsilon_{m}\varphi_{m}$ and $T_{4}=\sum_{m}\varepsilon_{m}^{2}\varphi_{m}$ through their Fourier representation yields the integral
\begin{align}
&\hspace{-4mm}V_{N}(\bm{t};\lambda)=\frac{2}{(2\pi)^{N}}\big[\prod_{n}(1+\frac{1}{\lambda_{n}})^{\frac{t_{n}}{2}}\big]\big[\prod_{k<l}\frac{\sqrt{(1+\lambda_{k})(1+\lambda_{l})}}{\sqrt{(1+\lambda_{k})(1+\lambda_{l})}-\sqrt{\lambda_{k}\lambda_{l}}}\big]\int\!\!\ud\tau_{1}\!\!\int\!\!\ud S_{1}\!\!\int\!\!\ud\tau_{3}\nonumber\\
&\times\int\!\!\ud T_{3}\!\!\int\!\!\ud T_{4}\!\!\int\!\!\ud \tau_{4}\!\!\int\!\!\ud S_{2}\!\!\int\!\!\ud\tau_{2}\,\exp[2\pi i(\tau_{1}S_{1}\!+\!\tau_{2}S_{2}\!+\!\tau_{3}T_{3}\!+\!\tau_{4}T_{4})\nonumber\\
&-A_{2}S_{1}^{2}-2B_{2}S_{1}T_{3}+2C_{2}S_{1}T_{4}-2D_{2}T_{3}^{2}+3A_{4}S_{2}^{2}]\nonumber\\
&\times\Big\{\prod_{j}\int_{-\delta_{N}}^{\delta_{N}}\!\!\!\!\!\ud\varphi_{j}\,\exp\big[i\varphi_{j}\big(-B_{1}N\varepsilon_{j}^{2}-B_{1}\sum_{m}\varepsilon_{m}^{2}-2\pi\tau_{1}-3A_{3}S_{2}\nonumber\\
&+C_{1}N\varepsilon_{j}^{3}-C_{1}\varepsilon_{j}\sum_{m}\varepsilon_{m}^{2}+C_{1}\sum_{m}\varepsilon_{m}^{3}-2\pi\tau_{3}\varepsilon_{j}\big)\big]\nonumber\\
&\times \exp\big[-\varphi_{j}^{2}\big(A_{2}(N-2)+B_{2}(N-4)\varepsilon_{j}-(N-4)C_{2}\varepsilon_{j}^{2}-C_{2}\sum_{m}\varepsilon_{m}^{2}+2\pi i\tau_{2}\big)\big]\nonumber\\
&\times \exp\big[-i\varphi_{j}^{3}\big(A_{3}(N-4)+4iA_{4}S_{1}\big)\big]\nonumber\\
&\times \exp\big[\varphi_{j}^{4}\big(A_{4}(N\!-\!8)\big)\big]\,\Big\}\label{e:labval1}\quad.
\end{align}
To ensure that that overall error consists of asymptotically irrelevant factors only, the $\varphi_{j}$-integral must be computed up to $\mathcal{O}(N^{-1})$. Dividing the integration parameter $\varphi_{j}$ by $\sqrt{A_{2}(N-2)}$ shows that the $\varphi_{j}$-integral is of the form
\begin{align}
&\hspace{-4mm}\frac{1}{\sqrt{A_{2}(N-2)}}\int_{-\delta_{N}\sqrt{A_{2}(N-2)}}^{\delta_{N}\sqrt{A_{2}(N-2)}}\!\!\!\!\!\!\ud\varphi\,\exp[\frac{i\varphi Q_{1}}{\sqrt{A_{2}(N-2)}}-\varphi^{2}Q_{2}-\frac{i\varphi^{3}Q_{3}}{(A_{2}(N-2))^{3/2}}+\frac{\varphi^{4}Q_{4}}{(A_{2}(N-2))^{2}}]\nonumber\\
&=\sqrt{\frac{\pi}{A_{2}(N-2)}}\big[Q_{2}+\frac{3iQ_{3}\tilde{\varphi}}{(A_{2}(N-2))^{3/2}}\big]^{-\frac{1}{2}}\nonumber\\
&\times\exp[\frac{iQ_{1}\tilde{\varphi}}{\sqrt{A_{2}(N-2)}}-Q_{2}\tilde{\varphi}^{2}-\frac{iQ_{3}\tilde{\varphi}^{3}}{(A_{2}(N-2))^{3/2}}+\frac{Q_{4}\tilde{\varphi}^{4}}{A_{2}^{2}(N-2)^{2}}]\nonumber\\
&\times\big\{1-\frac{15Q_{3}^{2}}{16(A_{2}(N-2))^{3}(Q_{2}+\frac{3iQ_{3}\tilde{\varphi}}{(A_{2}(N-2))^{3/2}})^{3}}+\frac{3Q_{4}}{4A_{2}^{2}(N-2)^{2}(Q_{2}+\frac{3iQ_{3}\tilde{\varphi}}{(A_{2}(N-2))^{3/2}})^{2}}\big\}\quad,\label{e:labval2}
\end{align}
which is calculated by the (\ref{e:spm}) around the maximum $\tilde{\varphi}$ of the integrand. Observing that $Q_{1}=\mathcal{O}(N^{2\omega})$, $Q_{2}=\mathcal{O}(1)$ and $Q_{3,4}=\mathcal{O}(N)$, shows that
\begin{equation*}
\tilde{\varphi}=\frac{iQ_{1}}{2Q_{2}\sqrt{A_{2}(N-2)}}+\mathcal{O}(N^{-\frac{3}{2}+4\omega})=\mathcal{O}(N^{-\frac{1}{2}+2\omega})
\end{equation*}
is sufficient for the desired accuracy. This implies that
\begin{equation*}
\exp[\frac{iQ_{1}\tilde{\varphi}}{\sqrt{A_{2}(N-2)}}-Q_{2}\tilde{\varphi}^{2}-\frac{iQ_{3}\tilde{\varphi}^{3}}{(A_{2}(N-2))^{3/2}}]=\exp\big[-\frac{Q_{1}^{2}}{4A_{2}(N-2)}\big]\times\big(1+\mathcal{O}(N^{-2+6\omega})\big)\quad.
\end{equation*}
The terms in square and curly brackets are then rewritten using
\begin{equation*}
\frac{1}{\sqrt{1+y}}\approx e^{-\frac{y}{2}+\frac{y^{2}}{4}}\qquad\text{and}\qquad 1+z\approx \exp[z]
\end{equation*}
respectively. Using the same order of factors as in (\ref{e:labval2}), the result of the $\varphi_{j}$-integral is
\begin{align*}
&\hspace{-4mm}\sqrt{\frac{\pi}{A_{2}(N-2)}}\exp\Big[-\frac{B_{2}(N-4)\varepsilon_{j}}{2A_{2}(N-2)}+\frac{C_{2}(N-4)\varepsilon_{j}^{2}}{2A_{2}(N-2)}+\frac{C_{2}\sum_{m}\varepsilon_{m}^{2}}{2A_{2}(N-2)}-\frac{i\pi\tau_{2}}{A_{2}(N-2)}\\
&-\frac{3A_{3}B_{1}N(N-4)\varepsilon_{j}^{2}}{4A_{2}^{2}(N-2)^{2}}-\frac{3A_{3}B_{1}(N-4)\sum_{m}\varepsilon_{m}^{2}}{4A_{2}^{2}(N-2)^{2}}-\frac{3\pi\tau_{1}A_{3}(N-4)}{2A_{2}^{2}(N-2)^{2}}-\frac{9A_{3}^{2}S_{2}(N-4)}{4A_{2}^{2}(N-2)^{2}}\Big]\\
&\times\exp\big[\frac{B_{2}^{2}(N-4)^{2}\varepsilon_{j}^{2}}{4A_{2}^{2}(N-2)^{2}}\big]\exp\Big[-\frac{B_{1}^{2}N^{2}\varepsilon_{j}^{4}}{4A_{2}(N-2)}-\frac{B_{1}^{2}N\varepsilon_{j}^{2}\sum_{m}\varepsilon_{m}^{2}}{2A_{2}(N-2)}\\
&-\frac{\pi\tau_{1}B_{1}N\varepsilon_{j}^{2}}{A_{2}(N-2)}-\frac{3A_{3}B_{1}S_{2}N\varepsilon_{j}^{2}}{2A_{2}(N-2)}-\frac{B_{1}^{2}(\sum_{m}\varepsilon_{m}^{2})^{2}}{4A_{2}(N-2)}-\frac{\pi\tau_{1}B_{1}\sum_{m}\varepsilon_{m}^{2}}{A_{2}(N-2)}-\frac{3A_{3}B_{1}S_{2}\sum_{m}\varepsilon_{m}^{2}}{2A_{2}(N-2)}\\
&-\frac{\pi^{2}\tau_{1}^{2}}{A_{2}(N-2)}-\frac{3\pi\tau_{1}A_{3}S_{2}}{A_{2}(N-2)}-\frac{9A_{3}^{2}S_{2}^{2}}{4A_{2}(N-2)}\Big]\exp\big[-\frac{15A_{3}^{2}(N-4)^{2}}{16A_{2}^{3}(N-2)^{3}}\big]\exp\big[\frac{3A_{4}(N-8)}{4A_{2}^{2}(N-2)^{2}}\big]\;.
\end{align*}
Integrating now $\tau_{2}$ yields a delta function that assigns the value
\begin{equation*}
S_{2}=\frac{N}{2A_{2}(N-2)}\quad.
\end{equation*}
Doing the same for $\tau_{3}$ and $\tau_{4}$ yields $T_{3}=0$ and $T_{4}=0$. The $S_{1}$-integral is
\begin{equation*}
\int \ud S_{1}\,\exp[2\pi i\tau_{1}S_{1}-A_{2}S_{1}^{2}]=\sqrt{\frac{\pi}{A_{2}}}\exp[-\frac{\pi^{2}\tau_{1}^{2}}{A_{2}}]\quad.
\end{equation*}
and the final integral
\begin{align*}
&\hspace{-4mm}\int \ud\tau_{1}\,\exp\big[-\frac{2\pi^{2}(N-1)\tau_{1}^{2}}{A_{2}(N-2)}-\frac{3\pi\tau_{1}A_{3}N}{A_{2}^{2}(N-2)}-\frac{2\pi\tau_{1}B_{1}N\sum_{m}\varepsilon_{m}^{2}}{A_{2}(N-2)}\big]\\
&=\sqrt{\frac{A_{2}(N-2)}{2\pi (N-1)}}\exp\big[\frac{9A_{3}^{2}N^{2}}{8A_{2}^{3}(N-1)(N-2)}+\frac{B_{1}^{2}N^{2}(\sum_{m}\varepsilon_{m}^{2})^{2}}{2A_{2}(N-1)(N-2)}+\frac{3A_{3}B_{1}N^{2}\sum_{m}\varepsilon_{m}^{2}}{2A_{2}^{2}(N-1)(N-2)}\big]\quad.
\end{align*}
Putting this all together yields
\begin{align}
&\hspace{-4mm}V_{N}(\bm{t};\lambda)=\frac{\sqrt{2}}{(2\pi\lambda(\lambda+1)N)^{\frac{N}{2}}}\big[\prod_{n}(1+\frac{1}{\lambda_{n}})^{\frac{t_{n}}{2}}\big]\big[\prod_{k<l}\frac{\sqrt{(1+\lambda_{k})(1+\lambda_{l})}}{\sqrt{(1+\lambda_{k})(1+\lambda_{l})}-\sqrt{\lambda_{k}\lambda_{l}}}\big]\nonumber\\
&\times\exp[\frac{14\lambda^{2}+14\lambda-1}{12\lambda(\lambda+1)}]\exp[\frac{\sum_{m}\varepsilon_{m}^{2}}{16\lambda^{2}(\lambda+1)^{2}}]\exp[-\frac{(2\lambda+1)^{2}}{128\lambda^{3}(\lambda+1)^{3}}(\sum_{m}\varepsilon_{m}^{2})^{2}]\nonumber\\
&\times\exp[-\frac{(2\lambda+1)^{2}N}{128\lambda^{3}(\lambda+1)^{3}}\sum_{m}\varepsilon_{m}^{4}]\quad.\label{e:labval3}
\end{align}
Comparing (\ref{e:labval3}) to the Lower bound $\mathcal{E}_{\alpha}$, it is immediately clear that $V_{n}(\bm{t};\lambda)$ is much larger. Expand the products in square brackets around $\lambda$,
\begin{align*}
&\hspace{-4mm}\Big[\prod_{j}(1+\frac{1}{\lambda_{j}})^{\frac{\lambda(N-1)+(N-2)\varepsilon_{j}/2}{2}}\Big]\Big[\prod_{k<l}\frac{\sqrt{(1+\lambda_{k})(1+\lambda_{l})}}{\sqrt{(1+\lambda_{k})(1+\lambda_{l})}-\sqrt{\lambda_{k}\lambda_{l}}}\Big]\\
&=(1+\frac{1}{\lambda})^{\frac{x}{2}}(1+\lambda)^{\binom{N}{2}}\exp[\frac{2\lambda(N-1)+(N-2)\varepsilon_{j}}{4}\log\big(\frac{1+\lambda+\varepsilon_{j}}{1+\lambda}\frac{\lambda}{\lambda+\varepsilon_{j}}\big)]\\
&\times\exp[-\sum_{k<l}\log\big(1+\lambda-\sqrt{(1+\lambda-\frac{1+\lambda}{1+\lambda+\varepsilon_{k}})(1+\lambda-\frac{1+\lambda}{1+\lambda+\varepsilon_{l}})}\big)]\\
&=(1+\frac{1}{\lambda})^{\frac{x}{2}}(1+\lambda)^{\binom{N}{2}}\\
&\times\exp\Big[(\sum_{m}\varepsilon_{m}^{2})\cdot\big[N\frac{\lambda^{2}}{4\lambda^{2}(\lambda+1)^{2}}+\frac{\lambda}{4\lambda^{2}(\lambda+1)^{2}}-(N-1)\frac{3\lambda^{2}+\lambda}{8\lambda^{2}(\lambda+1)^{2}}-\frac{1}{8\lambda(\lambda+1)}\big]\\
&+(\sum_{m}\varepsilon_{m}^{3})\cdot\big[-N\frac{6\lambda^{3}+3\lambda^{2}+\lambda}{24\lambda^{3}(\lambda+1)^{3}}+N\frac{14\lambda^{3}+9\lambda^{2}+3\lambda}{48\lambda^{3}(\lambda+1)^{3}}\big]\\
&+(\sum_{m}\varepsilon_{m}^{4})\cdot\big[N\frac{6\lambda^{4}+6\lambda^{3}+4\lambda^{2}+\lambda}{24\lambda^{4}(\lambda+1)^{4}}-N\frac{30\lambda^{4}+28\lambda^{3}+19\lambda^{2}+5\lambda}{128\lambda^{4}(\lambda+1)^{4}}\big]\\
&+(\sum_{m}\varepsilon_{m}^{2})^{2}\cdot\big[\frac{6\lambda^{2}+6\lambda+1}{128\lambda^{3}(\lambda+1)^{3}}\big]\times\big(1+\mathcal{O}(N^{-\frac{1}{2}+5\omega})\big)\quad,
\end{align*}
yields combined with (\ref{e:labval3}) the desired result.\\
To determine the error from the difference $\mathcal{D}$ from Lemma \ref{l:diff}, we divide it by $V_{N}(\bm{t};\lambda)$. Assuming that $|t_{j}-\lambda(N-1)|=\lambda N^{\frac{1}{2}+\omega}$ takes maximal values, it follows that the relative difference is at most
\begin{align*}
&\hspace{-4mm}\mathcal{O}(N^{2-K\alpha})\exp[\frac{4\lambda^{2}+4\lambda+1}{3\lambda(\lambda+1)}]\exp[\frac{4\lambda^{2}+4\lambda+3}{4(\lambda+1)^{2}}N^{2\omega}]\\
&\times\exp[\frac{(2\lambda+1)^{2} \lambda N^{4\omega}}{8(\lambda+1)^{3}}]\exp[\frac{(2\lambda+1)^{2} \lambda N^{4\omega}}{8(\lambda+1)^{3}}]\quad.
\end{align*}
Only the first exponential can become large, if $\lambda$ is small. Assuming that $\lambda > C/\log(N)$, this factor adds an error $N^{\frac{1}{3C}}$.\\
To keep this relative error small, it is furthermore necessary that $\exp[N^{4\omega}]\ll N^{K\alpha-2}$. Solving this yields
\begin{equation*}
0<\omega<\frac{\log(K\alpha-2)+\log\big(\log(N)\big)}{4\log(N)}\quad.
\end{equation*}
\end{proof}

Choosing the value of $\lambda$ may seem arbitrary at first. It is not. Comparing (\ref{e:labval3}) to the Lower bound $\mathcal{E}_{\frac{1}{2}-r}$ for some small $r>0$, the outcome is only much larger, if
\begin{equation*}
\sum_{m}\varepsilon_{m}^{4}\ll N^{2r}\qquad\text{and}\qquad \sum_{m}\varepsilon_{m}^{2}\ll N^{r}\quad.
\end{equation*}
It follows that $\lambda N(N-1)=x$ in the limit. In~\cite{mckay1} the number of matrices $V_{N}(\bm{t};\lambda)$ has been calculated for the case that all $t_{j}$ are equal. They require $\lambda$ to be the average matrix entry for infinitely large matrices. Because Lemma \ref{l:labval} covers this case too, the same value for $\lambda$ had to be expected.\\

Methods to treat such multi-dimensional combinatorical Gaussian integrals in more generality have been discussed in~\cite{mckay2}.

\section{Reduction of the integration region\label{sec:rir}}

In the previous paragraph the result of the integral (\ref{e:V3}) in a small box around the origin was obtained. Knowing this makes it much easier to compare the contribution inside and outside of this box. This is the main aim of Lemma \ref{l:funris}.

\begin{lemma}\label{l:est}
For $a\in[0,1]$ and $n\in\mathbb{N}$ the estimates
\begin{equation*}
\exp[na\log (2)]\leq (1+a)^{n}\leq \exp[na]
\end{equation*}
hold.
\end{lemma}
\begin{proof}
The right-hand side follows from
\begin{equation*}
(1+a)^{n}=\sum_{j=0}^{n}a^{j}\binom{n}{j}=\sum_{j=0}^{n}\frac{(na)^{j}}{j!}\frac{n!}{n^{j}(n-j)!}\leq \sum_{j=0}^{n}\frac{(na)^{j}}{j!}\leq \exp[na]\quad.
\end{equation*}
For the left-hand side it suffices to show that $\log(1+a)\geq a\log(2)$. Because equality holds at one and zero, this follows from the concavity of the logarithm.
\end{proof}

\begin{lemma}\label{l:funris}
For any $\omega\in(0,\frac{1}{4})$ and $\alpha\in(0,\frac{1}{4}-\omega)$, define
\begin{equation*}
\delta_{N}=\frac{N^{-\alpha}\zeta_{N}}{\min \{\lambda_{j}\}}\quad,
\end{equation*}
such that $\zeta_{N}\rightarrow\infty$ and $N^{-\delta}\zeta_{N}\rightarrow0$ for any $\delta>0$. Assuming that $x=\sum_{j}t_{j}=\lambda N(N-1)$, $|t_{j}-\lambda(N-1)|\ll \lambda N^{\frac{1}{2}+\omega}$ and $\lambda>C/log(N)$, the integral
\begin{align*}
&\hspace{-4mm}V_{N}(\bm{t})=\Big(\prod_{j=1}^{N}(1+\frac{1}{\lambda_{j}})^{\frac{t_{j}}{2}}\Big)(2\pi)^{-N}\int_{\mathbb{T}^{N}}\!\!\!\ud\bm{\varphi}\,e^{-i\sum_{j=1}^{N}\varphi_{j}t_{j}}\\
&\times\prod_{1\leq k<l\leq N}\!\!\frac{\sqrt{(1+\lambda_{k})(1+\lambda_{l})}}{\sqrt{(1+\lambda_{k})(1+\lambda_{l})}-\sqrt{\lambda_{k}\lambda_{l}}}\,\frac{1}{1-\frac{\sqrt{\lambda_{k}\lambda_{l}}}{\sqrt{(1+\lambda_{k})(1+\lambda_{l})}-\sqrt{\lambda_{k}\lambda_{l}}} (e^{i(\varphi_{k}+\varphi_{l})}-1)}
\end{align*}
can be restricted to
\begin{align*}
&\hspace{-4mm}V_{N}(\bm{t};\lambda)=\frac{2}{(2\pi)^{N}}\Big(\prod_{j=1}^{N}(1+\frac{1}{\lambda_{j}})^{\frac{t_{j}}{2}}\Big)\int_{[-\delta_{N},\delta_{N}]^{N}}\!\!\!\!\!\!\!\!\!\!\!\ud\bm{\varphi}\,\exp[-i\sum_{j}\varphi_{j}(t_{j}-\lambda(N-1))]\\
&\times\prod_{1\leq k<l\leq N}\!\!\frac{\sqrt{(1+\lambda_{k})(1+\lambda_{l})}}{\sqrt{(1+\lambda_{k})(1+\lambda_{l})}-\sqrt{\lambda_{k}\lambda_{l}}}\,\frac{1}{1-\frac{\sqrt{\lambda_{k}\lambda_{l}}}{\sqrt{(1+\lambda_{k})(1+\lambda_{l})}-\sqrt{\lambda_{k}\lambda_{l}}} (e^{i(\varphi_{k}+\varphi_{l})}-1)}\\
&\times\big(1+\mathcal{O}(N^{\frac{3}{2}}\exp[-N^{1-2\alpha}\zeta_{N}^{2}])\big)\quad.
\end{align*}
\end{lemma}

\begin{proof}
The idea of the proof is to consider the integrand in a small box $[-\delta_{N},\delta_{N}]^{N}$ and see what happens to it if some of the angles $\bm{\varphi}$ lie outside of it.\\
Because $x$ is even, it follows that the integrand takes the same value at $\bm{\varphi}$ and $\bm{\varphi}+\bm{\pi}=(\varphi_{1}+\pi,\ldots,\varphi_{N}+\pi)$. This means that only half of the space has to be considered and the result must be multiplied by $2$.\\

This estimate follows directly from application of (\ref{e:fracexpa}-\ref{e:comb}) to the integrand and a computation like the one in the proof of Lemma \ref{l:labval}. Writing
\begin{equation*}
\mu_{kl}=\frac{\sqrt{\lambda_{k}\lambda_{l}}}{\sqrt{(1+\lambda_{k})(1+\lambda_{l})}-\sqrt{\lambda_{k}\lambda_{l}}}\qquad\text{and}\qquad \varepsilon_{j}=\lambda_{j}-\lambda
\end{equation*}
with $|\varepsilon_{j}|\ll \lambda N^{-\frac{1}{2}+\omega}$ this yields
\begin{align}
&\hspace{-4mm}\Big|\int_{[-\delta_{N},\delta_{N}]^{N}}\!\!\!\!\!\!\!\!\ud\bm{\varphi}\,\prod_{1\leq k<l\leq N}\frac{1}{1-\mu_{kl}(\exp[i(\varphi_{k}+\varphi_{l})]-1)}\Big|\nonumber\\
&\leq \int_{[-\delta_{N}/2,\delta_{N}/2]^{N}}\!\!\!\!\!\!\!\!\ud\bm{\varphi}\,\big|\exp[\sum_{m=1}i^{m}\sum_{k<l}A_{m}(\mu_{kl})(\varphi_{k}+\varphi_{l})^{m}]\big|\nonumber\\
&\leq\sqrt{2}\big(\frac{2\pi}{\lambda(\lambda+1)N}\big)^{\frac{N}{2}}\exp[\frac{10\lambda^{2}+10\lambda+1}{4\lambda(\lambda+1)}]\exp[N^{\frac{1}{2}+2\omega}]\quad.\label{e:cl1}
\end{align}
The final exponent $\exp[N^{\frac{1}{2}+2\omega}]$ here comes from the estimate $\mu_{kl}\geq \lambda(1-N^{-\frac{1}{2}+\omega})$.\\

Now we argue case by case why other configurations of the angles $\varphi_{j}$ are asymptotically suppressed.\\
\emph{Case 1.} All but finitely many angles lie in the box $[-\delta_{N},\delta_{N}]^{N}$. A finite number of $m$ angles lies outside of it. We label these angles $\{\varphi_{1},\ldots,\varphi_{m}\}$. The maximum of the integrand
\begin{align*}
&\hspace{-4mm}f:(\varphi_{m+1},\ldots,\varphi_{N})\mapsto \prod_{1\leq k<l\leq N}\frac{1}{1-\mu_{kl}(\exp[i(\varphi_{k}+\varphi_{l})]-1)}
\end{align*}
in absolute value is given by the equations
\begin{equation*}
0=\partial_{\varphi_{j}}|f|=\sum_{k\neq j}\frac{\sin(\varphi_{j}+\varphi_{k})}{1-2\mu_{kl}(\mu_{kl}+1)(\cos(\varphi_{j}+\varphi_{k})-1)}\qquad\text{ for }j=m+1,\ldots,N\quad.
\end{equation*}
It is clear that the maximum is found for $\tilde{\varphi}=\varphi_{m+1}=\ldots=\varphi_{N}$. The first order solution to this is then
\begin{equation*}
\tilde{\varphi}=\frac{-1}{2(N-m-1)}\sum_{k=1}^{m}\frac{\sin(\varphi_{k})}{1+2\mu_{kj}(\mu_{kj}+1)(1-\cos\varphi_{k})}\quad.
\end{equation*}
This shows that the maximum will lie in the box $[-\delta_{N}/2,\delta_{N}/2]^{N}$. This implies that $|\varphi_{j}-\varphi_{k}|>\delta_{N}/2$, when $1\leq j\leq m$ and $m+1\leq k\leq N$. Applying the estimate (\ref{e:fracest}) to pairs of such angles and afterwards (\ref{e:cl1}) to the remaining $N-m$ angles in the box $[-\delta_{N},\delta_{N}]^{N-m}$ gives us an upper bound of
\begin{align*}
&\hspace{-4mm}\frac{2\sqrt{2}}{\big(2\pi\lambda(\lambda+1)(N-m)\big)^{\frac{N-m}{2}}}\Big(\prod_{j}(1+\frac{1}{\lambda_{j}})^{\frac{t_{j}}{2}}\Big)\cdot\Big(\prod_{k<l}\frac{\sqrt{(1+\lambda_{k})(1+\lambda_{l})}}{\sqrt{(1+\lambda_{k})(1+\lambda_{l})}-\sqrt{\lambda_{k}\lambda_{l}}}\Big)\\
&\times\binom{N}{m}\exp[\frac{30\lambda^{2}+30\lambda+3}{12\lambda(\lambda+1)}]\exp[N^{\frac{1}{2}+2\omega}](1+\frac{\lambda(\lambda+1)\delta_{N}^{2}}{4})^{-\frac{Nm}{2}}
\end{align*}
on the part of the integral in the small box $[-\delta_{N},\delta_{N}]^{N}$. There are $\binom{N}{m}$ ways to select the $m$ angles. Applying Lemma \ref{l:est} to the final factor and comparing the result with the Lower bound, shows that this may be neglected if
\begin{equation*}
2\sqrt{2}e^{\frac{16\lambda^{2}+16\lambda+4}{12\lambda(\lambda+1)}}e^{m/2}N^{\frac{3m}{2}}(2\pi\lambda(\lambda+1))^{\frac{m}{2}}\exp[N^{1-2\alpha}+N^{\frac{1}{2}+2\omega}]e^{-\frac{Nm\log(2)}{8}\lambda(\lambda+1)\delta_{N}^{2}}\rightarrow0\quad.
\end{equation*}
The condition $0<\alpha<\frac{1}{4}-\omega$ and the sequence $\zeta_{N}\rightarrow\infty$ guarantee this. In fact, the same argument works for all $m$ such that $m/N\rightarrow0$.\\

\emph{Case 2.} If the number $m=\rho N$ of angles outside the integration box $[-\delta_{N},\delta_{N}]^{N}$ increases faster, another estimate is needed, because the maximum $\tilde{\varphi}$ may lie outside of $[-\delta_{N}/2,\delta_{N}/2]$. It is clear that $0<\rho<1$ in the limit.\\
Estimate the location $\varphi_{j}=\tilde{\varphi}$ of the maximum is much trickier now. Regardless of its precise location, we will take the maximum value as the estimate for the integrand in the entire integration box. The smaller box $[-\delta_{N}/2,\delta_{N}/2]^{N}$ is considered once more. We distinguish two options.\\
\emph{-Case 2a.} The maximum lies in $[-\delta_{N}/2,\delta_{N}/2]^{N}$, thus $\tilde{\varphi}\in[-\delta_{N}/2,\delta_{N}/2]$.\\
Applying the estimate (\ref{e:fracest}) to this yields an upper bound
\begin{align*}
&\hspace{-4mm}\binom{N}{\rho N}(2\delta_{N})^{N(1-\rho)}(2\pi)^{\rho N}\Big(\prod_{j}(1+\frac{1}{\lambda_{j}})^{\frac{t_{j}}{2}}\Big)\\
&\times\Big(\prod_{k<l}\frac{\sqrt{(1+\lambda_{k})(1+\lambda_{l})}}{\sqrt{(1+\lambda_{k})(1+\lambda_{l})}-\sqrt{\lambda_{k}\lambda_{l}}}\Big)(1+\frac{1}{4}\lambda(\lambda+1)\delta_{N}^{2})^{-\frac{N^{2}\rho(1-\rho)}{4}}\quad.
\end{align*}
Applying Lemma \ref{l:est} to the last factor and dividing this by $\mathcal{E}_{\alpha}$ shows that
\begin{equation*}
\big(2^{\frac{1}{\rho}}\pi \delta^{\frac{1-\rho}{\rho}}(2\pi\lambda(\lambda\!+\!1)N)^{\frac{2}{\rho}}N\exp[\frac{N^{-2\alpha}}{\rho}+\frac{N^{-\frac{1}{2}+2\omega}(1\!-\!\rho)}{\rho}-\frac{\lambda(\lambda\!+\!1)\delta_{N}^{2}N(1\!-\!\rho)\log(2)}{16}]\big)^{\rho N}\rightarrow0
\end{equation*}
is a sufficient and satisfied condition.\\
\emph{-Case 2b.} The maximum lies not in $[-\delta_{N}/2,\delta_{N}/2]^{N}$. This is the same as $\delta_{N}/2<|\tilde{\varphi}|\leq\delta_{N}$.\\
Applying (\ref{e:fracest}) only to the angles $\varphi_{\rho N+1},\ldots,\varphi_{\rho N}$ in the integration box gives an upper bound
\begin{align*}
&\hspace{-4mm}\binom{N}{\rho N}(2\delta_{N})^{N(1-\rho)}(2\pi)^{\rho N}\Big(\prod_{j}(1+\frac{1}{\lambda_{j}})^{\frac{t_{j}}{2}}\Big)\\
&\times\Big(\prod_{k<l}\frac{\sqrt{(1+\lambda_{k})(1+\lambda_{l})}}{\sqrt{(1+\lambda_{k})(1+\lambda_{l})}-\sqrt{\lambda_{k}\lambda_{l}}}\Big)(1+\frac{1}{4}\lambda(\lambda+1)\delta_{N}^{2})^{-\frac{N^{2}(1-\rho)^{2}}{4}}\quad.
\end{align*}
The same steps as in Case 2a. will do.\\
This shows that the integration can be restricted to the box $[-\delta_{N},\delta_{N}]^{N}$. The error terms follow from \emph{Case 1.}, since convergence there is much slower.
\end{proof}

Lemma \ref{l:funris} shows that for every $\alpha\in(0,1/4-\omega)$ and $N\in\mathbb{N}$ there is a box that contains most of the integral's mass. As $N$ increases, this box shrinks and the approximation becomes better. The parameter $\alpha$ determines how fast this box shrinks. Smaller values of $\alpha$ lower the Lower bound and, hence, increase the number of configurations within reach at the price of more intricate integrals and less accuracy.\\

The observation that $\zeta_{N}=\log(N)$ and $K\geq \alpha^{-1}(\frac{3}{2}+\frac{1}{3C}-6\omega)$ satisfies all the demands proves Theorem \ref{thrm:p1}.\\
An idea of the accuracy of these formulas can be obtained from Table \ref{t:t1} and \ref{t:t2}, where the reference values
\begin{equation}
y_{k}=\sum_{j=1}^{N}(t_{j}-\lambda(N-1))^{k}\qquad\text{for }k\geq 2\label{e:refval}
\end{equation}
are defined to compare configurations with the reference values $2^{-k}\lambda^{k}N^{1+\frac{k}{2}}$ for $k\geq 2$.

\begin{table}[!hb]\begin{footnotesize}
\hspace{5mm}\begin{tabular}{ccccccc||c||ccc||c||c}
$t_{1}$ & $t_{2}$ & $t_{3}$ & $t_{4}$ & $t_{5}$ & $t_{6}$ & $t_{7}$ & $\#$ & $y_{2}$ & $y_{3}$ & $y_{4}$ & $V_{N}(\bm{t;\lambda})$ & ratio \\\hline
8&8&8&8&8&8&8&5.42E7&0&0&0&5.03E7&0.928\\
7&8&8&8&8&8&9&5.07E7&2&0&2&4.74E7&0.935\\
7&7&8&8&8&9&9&4.75E7&4&0&4&4.47E7&0.941\\
7&7&7&8&9&9&9&4.45E7&6&0&6&4.21E7&0.947\\
6&8&8&8&8&8&10&4.15E7&8&0&32&3.96E7&0.955\\
6&7&8&8&9&9&9&4.13E7&8&-6&20&3.94E7&0.953\\
7&7&7&8&8&9&10&4.18E7&8&6&20&4.00E7&0.956\\
5&8&8&8&9&9&9&3.53E7&12&-24&84&3.40E7&0.964\\
7&7&7&8&8&8&11&3.71E7&12&24&84&3.62E7&0.976\\
5&7&8&8&9&9&10&3.12E7&16&-18&100&3.05E7&0.977\\
6&7&7&7&9&10&10&3.23E7&16&6&52&3.16E7&0.977\\
7&7&7&7&8&8&12&2.91E7&20&60&260&2.96E7&1.017\\
5&5&5&9&10&11&11&1.08E7&50&-18&422&1.11E7&1.031\\
5&7&7&7&7&9&14&1.17E7&50&186&1382&1.34E7&1.143\\
4&6&7&7&8&10&14&7.92E6&62&150&1586&8.94E6&1.128
\end{tabular}\end{footnotesize}
\caption{\small The number ($\#$) of symmetric $7\times 7$-matrices with zero diagonal and natural entries summing to $x=56$ such that the $j$-th row sums to $t_{j}$ and the asymptotic estimates for this number by $V_{N}(\bm{t};\lambda)$ from Lemma \ref{l:labval} with $\lambda=x/(N(N-1))$ the average matrix entry. The parameters $y_{2}$, $y_{3}$ and $y_{4}$ are defined in (\ref{e:refval}) and their reference values are are $22$, $38$ and $68$ respectively. The convergence condition is $|t_{j}-\lambda(N-1)|\leq1.3$. The notation $1.0E6=1.0\times 10^{6}$ is used here.\label{t:t1}}
\end{table}
 
\begin{table}[!hb]\begin{footnotesize}
\hspace{25mm}\begin{tabular}{ccc|c|c|c}
$N$ & $t$ & $\lambda$ & $\#$ & $V_{N}(t;\lambda)$ & ratio\\\hline
6&6&1.20&3.69E4&3.34E4&0.906\\
7&8&1.33&5.42E7&5.03E7&0.928\\
8&9&1.29&1.10E11&1.04E11&0.938\\
9&10&1.25&8.46E14&8.00E14&0.946\\
10&11&1.22&2.45E19&2.34E19&0.952\\
11&12&1.20&2.71E24&2.60E24&0.957\\
12&13&1.18&1.14E30&1.10E30&0.961\\
13&14&1.17&1.86E36&1.79E36&0.965\\
14&15&1.15&1.16E43&1.12E43&0.968\\
15&14&1.00&6.36E46&6.18E46&0.971\\
16&12&0.80&6.32E47&6.15E47&0.974\\
17&12&0.75&9.55E52&9.32E52&0.976\\
18&12&0.71&2.02E58&1.97E58&0.978
\end{tabular}\end{footnotesize}
\caption{\small The number ($\#$) of symmetric $N\times N$-matrices~\cite{mckay4} with zero diagonal and natural entries, such that each row sums to $t$ and the asymptotic estimates for these numbers by $V_{N}(t;\lambda)$ Lemma \ref{l:labval}, where $\lambda=x/(N(N-1))$ the average matrix entry. The notation $1.0E6=1.0\times 10^{6}$ is used here.\label{t:t2}}
\end{table}

\begin{figure}[!hbt]
\begin{subfigure}[t]{0.5\linewidth}
\includegraphics[width=0.9\textwidth]{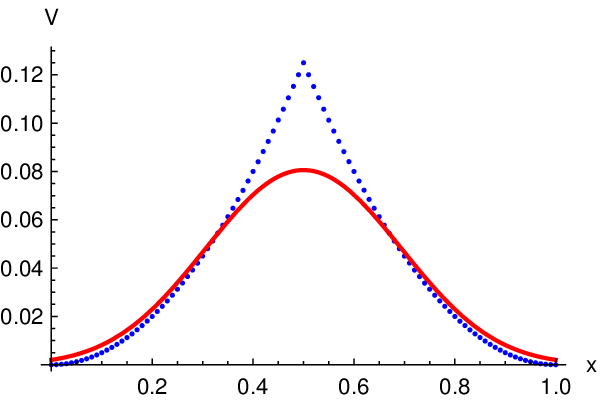}
\caption{N=4\label{f:f1a}}
\end{subfigure}%
\begin{subfigure}[t]{0.5\linewidth}
\includegraphics[width=0.9\textwidth]{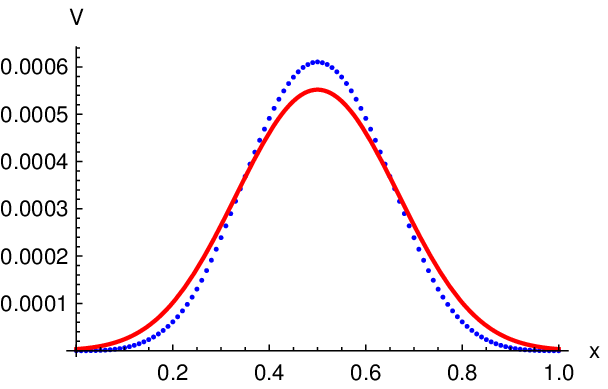}
\subcaption{N=5\label{f:f1b}}
\end{subfigure}\\
\begin{subfigure}[t]{0.5\linewidth}
\includegraphics[width=0.9\textwidth]{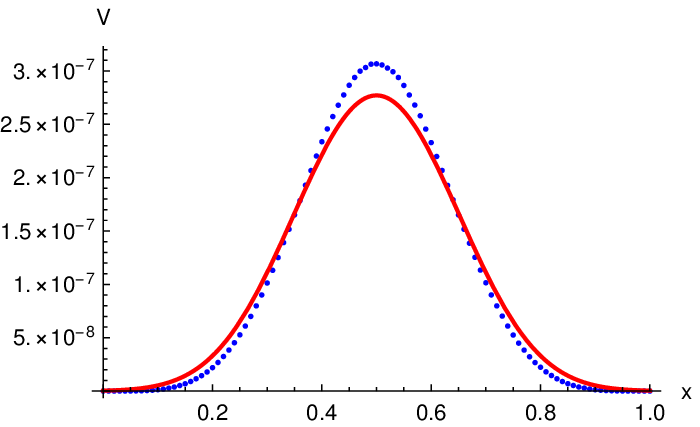}
\caption{N=6\label{f:f1c}}
\end{subfigure}%
\begin{subfigure}[t]{0.5\linewidth}
\includegraphics[width=0.9\textwidth]{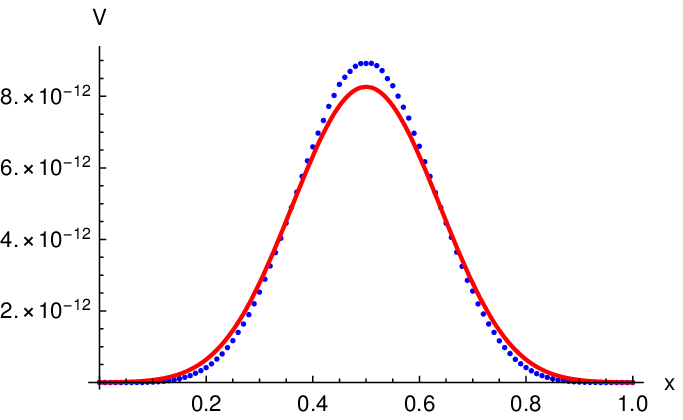}
\subcaption{N=7\label{f:f1d}}
\end{subfigure}\\
\begin{subfigure}[t]{0.5\linewidth}
\includegraphics[width=0.9\textwidth]{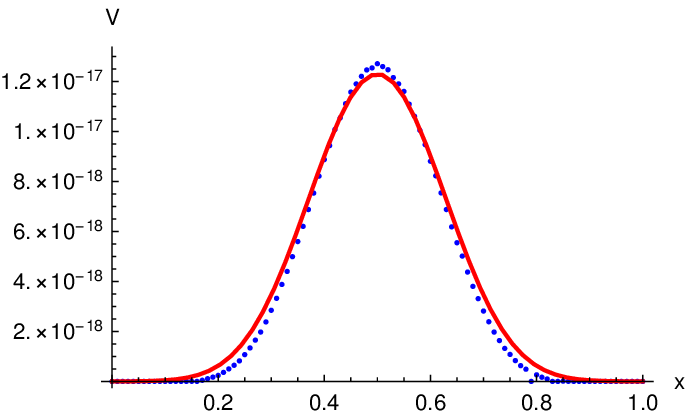}
\caption{N=8\label{f:f1e}}
\end{subfigure}%
\begin{subfigure}[t]{0.5\linewidth}
\includegraphics[width=0.9\textwidth]{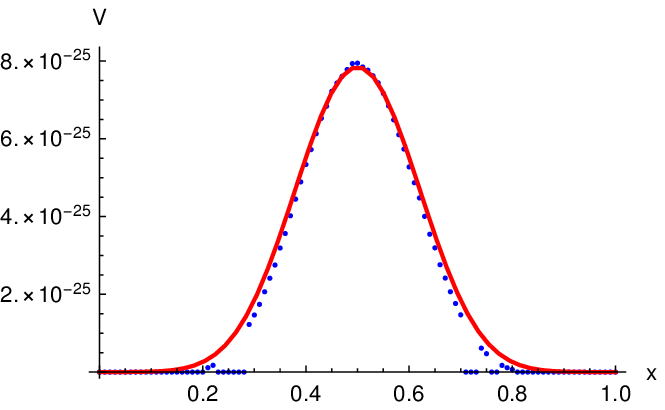}
\subcaption{N=9\label{f:f1f}}
\end{subfigure}
\caption{\small The volume result (\ref{e:pv1}) (red) and the volume of $P_{N}(0.5,\ldots,0.5,x,1-x)$ (blue) for $N=4,5,6,7,8,9$. The latter were determined by a numerical integration algorithm for convex multidimensional step functions on the basis of Monte Carlo integration.\label{f:f1}}
\end{figure}

\begin{figure}[!hbt]
\begin{subfigure}[t]{0.5\linewidth}
\includegraphics[width=0.9\textwidth]{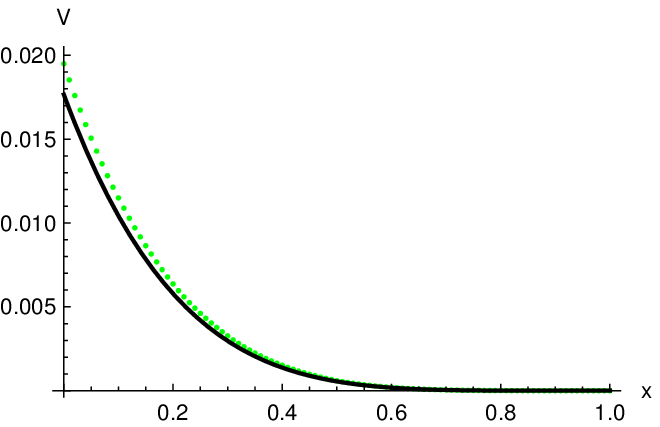}
\caption{$\vol\big(P_{5}(x,x,x,x,x)\big)$\label{f:f2a}}
\end{subfigure}%
\begin{subfigure}[t]{0.5\linewidth}
\includegraphics[width=0.9\textwidth]{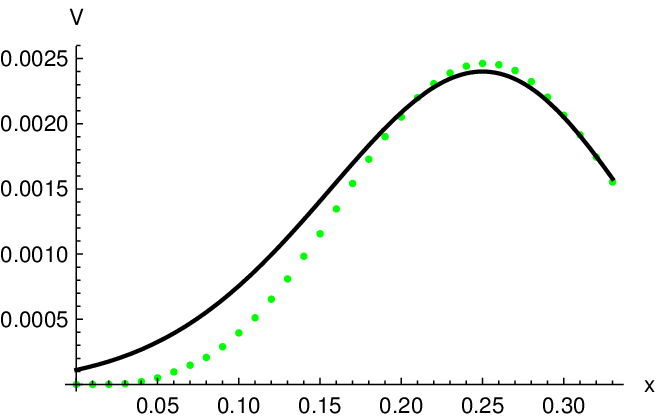}
\subcaption{$\vol\big(P_{5}(0.5,x,x,x,1-3x)\big)$\label{f:f2b}}
\end{subfigure}
\caption{\small The volume result (\ref{e:pv1}) (black) and the volume for two functions (green) for $N=5$. The latter were determined by straightforward Monte Carlo integration.\label{f:f2}}
\end{figure}

\section{Polytopes\label{sec:pv}}
In the previous paragraphs the asymptotic counting of symmetric matrices with zero diagonal and entries in the natural numbers was discussed. This allows us to return to the polytopes. The first step is to count the total number of symmetric matrices with zero diagonal and integer entries summing up to $x$ to see which fraction of such matrices are covered by Theorem \ref{thrm:p1}. This is easily done by a line of $\binom{N}{2}+\frac{x}{2}$ elements, for example unit elements $1$, and $\binom{N}{2}-1$ semicolons. Putting the semicolons between the elements, such that the line begins and ends with a unit element and no semicolons stand next to each other, creates such a matrix. The number of elements before the first semicolon minus one is the first matrix element $b_{12}$. The number of elements minus one between the first and second semicolon yields the second matrix element $b_{13}$. In this way, we obtain the $\binom{N}{2}$ elements of the upper triangular matrix. There are $\binom{N}{2}-1+\frac{x}{2}$ positions to put $\binom{N}{2}-1$ semicolons and thus
\begin{equation}
\binom{\binom{N}{2}-1+\frac{x}{2}}{\binom{N}{2}-1}\approx \frac{1}{N}\sqrt{\frac{1}{\pi\lambda(\lambda+1)}}(1+\lambda)^{\binom{N}{2}}(1+\frac{1}{\lambda})^{\frac{x}{2}}\big(1+\mathcal{O}(N^{-1})\big)\label{e:nosm}
\end{equation}
such matrices, where we have used Stirling's approximation and the average matrix entry condition $\lambda=x/(N(N-1))$ for the approximation.\\

The next step is to estimate the number of matrices within reach of Theorem \ref{thrm:p1}. Using only the leading order, the number of covered matrices is given by
\begin{align}
&\hspace{-4mm}\int \ud \bm{t}\;V_{N}(\bm{t};\lambda)\,\delta(\lambda N(N-1)-\sum_{j}t_{j})\nonumber\\
&=\frac{(1+\lambda)^{\binom{N}{2}}(1+\frac{1}{\lambda})^{\frac{x}{2}}}{\pi^{\frac{N}{2}}\sqrt{\lambda (\lambda+1)N}}\exp[\frac{14\lambda^{2}+14\lambda-1}{12\lambda(\lambda+1)}]\!\int\!\!\ud S\!\! \int\!\!\ud\sigma\!\!\int\!\!\ud\tau\,\exp[2\pi i\sigma S+S^{2}]\nonumber\\
&\times\Big\{\prod_{j=1}^{N}\int_{-N^{\omega}}^{N^{\omega}}\!\!\!\!\ud y_{j}\,\exp[2\pi i\tau y_{j}]\exp[-y_{j}^{2}(1\!+\!\frac{2}{N}\!-\!\frac{2\pi i\sigma}{N})]\exp[\frac{\sqrt{2}(2\lambda+1)y_{j}^{3}}{3\sqrt{\lambda(\lambda+1)N}}]\nonumber\\
&\times\exp[-\frac{3\lambda^{2}+3\lambda+1}{3\lambda(\lambda+1)N}y_{j}^{4}]\Big\}\nonumber\\
&=\frac{(1+\lambda)^{\binom{N}{2}}(1+\frac{1}{\lambda})^{\frac{x}{2}}}{N\sqrt{\pi\lambda (\lambda+1)}}\exp[\frac{14\lambda^{2}+14\lambda-1}{12\lambda(\lambda+1)}]\!\int\!\!\ud S\!\! \int\!\!\ud\sigma\!\!\int\!\!\ud\tau\,\exp[2\pi i\sigma S+S^{2}]\nonumber\\
&\times\exp[-\pi^{2}\tau^{2}(1\!-\!\frac{2}{N}\!+\!\frac{2i\pi \sigma }{N})]\exp[-1+\pi i \sigma+\frac{\pi i\tau(2\lambda+1)}{\sqrt{2\lambda(\lambda+1)}}]\exp[\frac{5(2\lambda+1)^{2}}{24\lambda (\lambda+1)}]\nonumber\\
&\times\exp[-\frac{3\lambda^{2}+3\lambda+1}{4\lambda(\lambda+1)}]\times\big(1-\mathcal{O}(\frac{\exp[-N^{2\omega}]}{N^{2\omega}})\big)^{N}\nonumber\\
&=\frac{(1+\lambda)^{\binom{N}{2}}(1+\frac{1}{\lambda})^{\frac{x}{2}}}{N\sqrt{\pi\lambda (\lambda+1)}}\exp[-\frac{1}{4\lambda(\lambda+1)}]\times\big(1-\mathcal{O}(\frac{\exp[-N^{2\omega}]}{N^{2\omega}})\big)^{N}\times\big(1+\mathcal{O}(N^{-1})\big)\quad.\label{e:matcov}
\end{align}

A fraction $\exp[-\frac{1}{4\lambda(\lambda)}]$ of the matrices is covered, provided that $\omega$ is large enough. A sufficient condition is that 
\begin{equation}
\omega \geq \frac{\log \log N}{2\log N}\quad.
\end{equation}
Combining this with the condition 
\begin{equation*}
\omega \leq \frac{\log (K\alpha-2) + \log(\log N)}{4\log (N)}
\end{equation*}
shows that $K\geq \log(N)/\alpha+2$  is necessary to satisfy both demands. However, such large values of $K$ remain without consequences, because higher values of $K$ only influence the the error term in Lemma \ref{l:labval}.\\
As $\lambda\rightarrow\infty$, the fraction of covered matrices tends to one and the volume of the diagonal subpolytopes of symmetric stochastic matrices can be determined by (\ref{e:ehrhart}). In terms of the variables
\begin{equation*}
t_{j}=\frac{1-h_{j}}{a}\quad\text{and}\quad \chi=\sum_{j}h_{j}
\end{equation*}
the volume of the diagonal subpolytope is calculated by
\begin{align}
&\hspace{-4mm}\vol(P_{N}(\bm{h}))=\lim_{a\rightarrow0}a^{\frac{N(N-3)}{2}}V_{N}(\frac{\bm{1}-\bm{h}}{a};\frac{N-\chi}{aN(N-1)})\nonumber\\
&=\sqrt{2}e^{\frac{7}{6}}\Big(\frac{e(N-\chi)}{N(N-1)}\Big)^{\binom{N}{2}}\Big(\frac{N(N-1)^{2}}{2\pi(N-\chi)^{2}}\Big)^{\frac{N}{2}}\exp[-\frac{N(N-1)^{2}}{2(N-\chi)^{2}}\sum_{j}(h_{j}-\frac{\chi}{N})^{2}]\nonumber\\
&\times\exp[-\frac{(N-1)^{2}}{(N-\chi)^{2}}\sum_{j}(h_{j}-\frac{\chi}{N})^{2}]\exp[-\frac{N(N-1)^{3}}{3(N-\chi)^{3}}\sum_{j}(h_{j}-\frac{\chi}{N})^{3}]\nonumber\\
&\times\exp[-\frac{N(N-1)^{4}}{4(N-\chi)^{4}}\sum_{j}(h_{j}-\frac{\chi}{N})^{4}]\exp[\frac{(N-1)^{4}}{4(N-\chi)^{4}}\big(\sum_{j}(h_{j}-\frac{\chi}{N})^{2}\big)^{2}]\quad.\label{e:pv1}
\end{align}

The convergence criterion becomes
\begin{equation*}
\frac{|t_{j}-\lambda (N-1)|}{\lambda N^{\frac{1}{2}+\omega}}=\frac{N^{\frac{1}{2}-\omega}(N-1)}{N-\chi}|h_{j}-\frac{\chi}{N}|\rightarrow0\quad.
\end{equation*}
This is the same as
\begin{equation*}
\sum_{j}|h_{j}-\frac{\chi}{N}|^{k}\ll (\frac{N-\chi}{N-1})^{k}N^{1-\frac{k}{2}+k\omega}\quad\text{for all }k\geq2\quad.
\end{equation*}

This means that we only have accuracy in a small neighbourhood around $\bm{\chi/N}$. However, the calculation (\ref{e:matcov}) shows that this corresponds to almost all matrices asymptotically, so that outside of this region the polytopes will have very small volumes. There, not all relevant factors are known, but missing factors will be small compared to the dominant factor. This means that for diagonals that satisfy
\begin{equation*}
\lim_{N\rightarrow\infty}\frac{(N-1)^{2}\sum_{j}(h_{j}-\frac{\chi}{N})^{2}}{(N-\chi)^{2}\log(N)}=0
\end{equation*}
qualitatively reasonable results are expected.\\

Since we are calculating a $\binom{N}{2}$-dimensional volume with only one length scale, it follows that no correction can become large in this limit. It inherits the relative error from Theorem \ref{thrm:p1}. This proves Theorem \ref{thrm:p2}. Examples of this formula at work are given in Figure \ref{f:f1} and \ref{f:f2}. 

\subsection*{Acknowledgments}
This work was supported by the Deutsche Forschungsgemeinschaft (SFB 878 - groups, geometry \& actions). We thank N. Broomhead and L. Hille for valuable discussions. In particular, we thank B.D. McKay and M. Isaev for clarifications of various aspects in asymptotic enumeration, pounting out a gap in the previous version and their suggestions to fill it.

\bibliography{Art2.bib}{}
\bibliographystyle{unsrt}

\end{document}